\documentclass[smallextended,referee,envcountsect]{svjour3}
\usepackage{booktabs}
\usepackage{amscd}
\smartqed
\usepackage{graphicx,amssymb}
\usepackage{enumerate,amsfonts}
\usepackage{amsmath}
\usepackage[bookmarksnumbered, plainpages, colorlinks]{hyperref}
\usepackage{multicol,multirow}
\usepackage{subfigure}
\usepackage{threeparttable}
\usepackage[top=2cm, bottom=2cm, left=2cm, right=2cm]{geometry}
\usepackage{float}
\usepackage[section]{placeins}
\graphicspath{{figures/}}
\usepackage[ruled,vlined]{algorithm2e} %三线排版
\usepackage{bigstrut}
\usepackage[numbers,sort&compress]{natbib}

\numberwithin{equation}{section}
\usepackage{mathptmx}

\begin{document}
\baselineskip=16pt

\title{Several self-adaptive inertial projection algorithms for solving split variational inclusion problems}

\titlerunning{Several self-adaptive inertial projection algorithms}

\author{Zheng Zhou $^*$ \and Bing Tan \and Songxiao Li}

\authorrunning{Z. Zhou, B. Tan, S. Li}

\institute{$^*$Corresponding author.\at
 Zheng Zhou \at
\email{zhouzheng2272@163.com}\\
Bing Tan \at
\email{bingtan72@gmail.com}\\
Songxiao Li \at
\email{jyulsx@163.com}
\and
 Institute of Fundamental and Frontier Sciences, University of Electronic Science and Technology of China, Chengdu 611731, China}

\date{Received: date / Accepted: date}

\maketitle

\begin{abstract}
This paper is to analyze the  approximation solution of a split variational inclusion problem in the framework of infinite dimensional Hilbert spaces. For this purpose, several inertial hybrid and shrinking projection algorithms are proposed under the effect of self-adaptive stepsizes which does not require information of the norms of the given operators. Some strong convergence properties of the proposed algorithms are obtained under mild constraints. Finally, an experimental application is given to illustrate the performances of proposed methods by comparing existing results.

\end{abstract}

\section{Introduction}

Inspired by the split variational inequality problem proposed by Censor et al. \cite{censor2010}, Moudafi \cite{moudafi2011split} introduced a more general form of this problem, that is, the split monotone variational inclusion problem (for short, SMVIP). It is worth noting that an important special case of the split monotone variation inclusion problem is the split variational inclusion problem (for short, SVIP), which is to find a zero of a maximal monotone mapping in one space, and the image of which under a given bounded linear transformation is a zero of another maximal monotone mapping in another space. As well as, the split variational inclusion problem is also a generalized form of many problems, such as the split variational inequality problem, the split minimization problem, the split equilibrium problem, the split saddle point problem and the split feasibility problem; see, for instance,  \cite{moudafi2011split,byrne2011weak,long2019new,qin2019,anh2020astrong} and the references therein. As applications, these problems are also widely applied to radiation therapy treatment planning, image recovery and signal recovery. For detail, we refer to \cite{zhou2020anew,chambolle1997image,nikolova2004variational}. In SVIP, when the two spaces are the same and the given bounded linear operator is an identity mapping, SVIP is equivalent to the well-known common solution problem, i.e., the common solution of two variational inclusion problems. Naturally, common solution problems of other aspects can be obtained, such as the variational inequality problem, the minimization problem and the equilibrium problem. In general, the above common solution problems can be regarded as the distinguished convex feasibility problem.

In particular, finding the zero of a maximal monotone mapping is known as the variational inclusion problem (for short, VIP), which is a special case of the SVIP. Furthermore, the resolvent mapping of the maximal monotone mapping is considered to solve the approximate solution of VIP. With the help of this resolvent mapping and the attention of a large number of scholars, the variational inclusion problem and the split variational inclusion problem has obtained quite a few remarkable results; see, e.g., \cite{qin2014var,chuang2016algorithms,kazmi2014an,Sitthithakerngkiet2018,cho2014st}, etc. On the other hand, based on the idea of the implicit discretization of a differential system of the second-order in time, Alvarez and Attouch \cite{alvarez2001} introduced an inertial proximal point algorithm to approximate a solution of the VIP. Under the effect of the inertial technique, the iterative sequence of SVIP and other problems rapidly converges to the approximation solution of the corresponding problems, such as the split variational inclusion problem \cite{long2019new,qin2019,anh2020astrong}, the split common fixed point problem \cite{zhou2020anew,zhou2020an}, the monotone inclusion problem \cite{lorenz2015an,qin2014v,tan2020strong}.

From the existing results of the split variational inclusion problem, we find that it is easy to get the weak convergence property, and sometimes its strong convergence is proved in the case of other methods, such as the viscosity method, the Halpern method, the Mann-type method, the hybrid steepest descent method, and so on, for detail, see \cite{byrne2011weak,long2019new,anh2020astrong,Sitthithakerngkiet2018}. Unfortunately, the stepsize sequences in these existing results often depends on the norm of the bounded linear operator. Hence, the work of this paper can be summarized in two aspects. The first one is to construct some new inertial iterative algorithms that converge strongly to a solution of SVIP. For this purpose, we consider two projection methods in our algorithms, namely hybrid projection \cite{NAKAJO2003372} and shrinking projection \cite{TAKAHASHI2008276}. The second one is to design a new stepsize sequence which does not need prior knowledge of the bounded linear operator in our algorithms.

The rest of the article is outlined as follows. Section \ref{sec2} introduces the split variational inclusion problem and some preliminaries. Several new iterative algorithms and their convergence theorems for SVIP are proposed in Section \ref{sec3}. Theoretical applications on other mathematical problems are given in Section \ref{sec4}. Finally, in Section \ref{sec5}, the validity and authenticity of the convergence behavior of the proposed algorithms are demonstrated by some applicable numerical examples.

\section{State of problem and preliminaries}\label{sec2}
\subsection{Split variational inclusion problem}
Let $H_1$ and $H_2$ be Hilbert spaces, $B_1:H_1\rightarrow 2^{H_1}$ and $B_2:H_2\rightarrow 2^{H_2}$ be maximal monotone mappings. Let $f_1:H_1\rightarrow H_1$ and $f_2:H_2\rightarrow H_2$ be single-valued mappings, $A:H_1\rightarrow H_2$ be a bounded linear operator. The split monotone variational inclusion problem is to find a point $x^*\in H_1$ such that
\begin{align}\label{smvip}
0\in f_1(x^*)+B_1(x^*)\ \text{ and }\ \ 0\in f_2(Ax^*)+B_2(Ax^*).\tag{SMVIP}
\end{align}
When $f_1\equiv 0$ and $f_2\equiv 0$, SMVIP can be considered as the split variational inclusion problem, which is to find a point $x^*\in H_1$ such that
\begin{align}\label{svip}
0\in B_1(x^*)\ \text{and}\ 0\in B_2(Ax^*).\tag{SVIP}
\end{align}
The solution set of SVIP is denoted by $\Omega$, i.e., $\Omega := \{x^*\in H_1: 0\in B_1(x^*)\ \text{and}\ 0\in B_2(Ax^*)\}$.

\subsection{Preliminaries}
To standardize, the notations $\rightarrow$ and $\rightharpoonup$ stand for strong convergence and weak convergence, respectively. The symbol $Fix(S)$ denotes the fixed point set of a mapping $S$. $\omega_w(x_n)$ represents the set of weak cluster point of a sequence $\{x_n\}$. Let $H$ be a Hilbert space with the inner product $\langle \cdot ,\cdot \rangle$ and the norm $\|\cdot \|$. Let $B:H\rightarrow 2^H$ be a set-valued mapping with domain $\mathcal{D}(B)=\{x\in H: B(x)\neq \emptyset \}$ and graph $\mathcal{G}(B)=\{(x,w)\in H\times H: x\in \mathcal{D}(B), w\in B(x)\}$. Recall that a mapping $B:H\rightarrow 2^H$ is monotone if and only if $\langle x-y, w-v\rangle\geq 0$ for any $w\in B(x)$ and $v\in B(y)$. Further, a monotone mapping $B:H\rightarrow 2^H$ is maximal, that is, the graph $\mathcal{G}(B)$ is not properly contained in the graph of any other monotone mapping. In this case, $B$ is a maximal monotone mapping if and only if for any $(x,w)\in \mathcal{G}(B)$ and $(y,v)\in H\times H$, $\langle x-y, w-v \rangle\geq 0$ implies $v\in B(y)$.

\begin{definition}
The mapping $S:H\rightarrow H$ is said to be
\begin{description}
\item[(I)] nonexpansive if $\|Sx-Sy\|\leq \|x-y\|$, $\forall x,y\in H$;
\item[(II)]  firmly nonexpansive if $\|Sx-Sy\|^2\leq \langle Sx-Sy, x-y\rangle$, $\forall x,y\in H$.
\end{description}
\end{definition}
\begin{remark}\label{remark21}
When $S$ is a firmly nonexpansive mapping, it is also nonexpansive and $I-S$ is also a firmly nonexpansive mapping.
\end{remark}
\begin{lemma}\cite{marino2004convergence,takahashi2000}
The resolvent mapping $J_\beta^{B}$ of a maximal monotone mapping $B$ with $\beta>0$ is defined as $J_\beta^{B}(x)=(I+\beta B)^{-1}(x), \forall x\in H$. The following properties associated with $J_\beta^{B}$ hold.
\begin{description}
\item[(1)] The mapping $J_\beta^{B}$ is a single-valued and firmly nonexpansive;
\item[(2)] The fixed point set of $J_\beta^{B}$ is equivalent to $B^{-1}(0)=\{x\in \mathcal{D}(B): 0\in B(x)\}$.
\end{description}
\end{lemma}
\begin{definition}
The notation $P_C$ denotes the metric projection from $H$ onto $C$, that is,  $P_Cx=\operatorname{argmin}_{y\in C} \|x-y\|,\ \forall x\in H$. Naturally, we can know the following equivalent properties of $P_C$:
\[
\langle P_Cx-x, P_Cx-y\rangle \le 0, \forall y \in C \Leftrightarrow \|y-P_{C} x\|^{2}+\|x-P_{C} x\|^{2} \leq\|x-y\|^{2}.
\]
\end{definition}

\begin{lemma}\cite{takahashi2000}\label{lem2}
Let $C$ be a nonempty closed convex subset of $H$ and $S:C\rightarrow C$ be a nonexpansive mapping with $Fix(S)\neq\emptyset$. $I-S$ is demiclosed at zero, that is, for any sequence $\{x_n\}$ in $C$, satisfying $x_n\rightharpoonup x$ and $(I-S)x_n\rightarrow0$, then $x\in Fix(S)$.
\end{lemma}

\begin{lemma}\cite{MARTINEZYANES20062400}\label{lem3}
Let $C$ be a nonempty closed convex subset of $H$. Let a sequence $\{x_n\}$ in $H$ and $u=P_Cv,\ v\in H$. If $\omega_w(x_n)\subset C$ and $\|x_n-v\|\leq\|u-v\|$, then $\{x_n\}$ converges strongly to $u$.
\end{lemma}

\section{Several self-adaptive inertial hybrid and shrinking projection algorithms}\label{sec3}
Combining the inertial technique with the projection methods, two types of projection algorithms are given for approximating a solution of the split variational inclusion problem \eqref{svip}. Before this, we always assume that the following conditions are satisfied:
\begin{description}
	\item[\textbf{(C1)}] $H_1$, $H_2$ are two Hilbert spaces and $A:H_1\rightarrow H_2$ is a bounded linear operator with adjoint operator $A^*$;
	\item[\textbf{(C2)}] $B_1: H_1\rightarrow 2^{H_1}$ and $B_2: H_2\rightarrow 2^{H_2}$ are two set-valued maximal monotone mappings.
\end{description}

Inertial hybrid projection algorithms and inertial shrinking projection algorithms are introduced below and the strong convergence of these algorithms are guaranteed by the following appropriate parameter conditions:
\begin{description}
	\item[\textbf{(P1)}] $\{\alpha_n\}\subset [a, b]\subset (-\infty, \infty)$ and $\{\beta_n\}\subset(0,\infty)$;
	\item[\textbf{(P2)}] If $Az_n\notin B_2^{-1}0$, the stepsize $\gamma_n=\frac{\sigma_n\|(I-J_{\beta_n}^{B_2})Az_n\|^2}{\|A^*(I-J_{\beta_n}^{B_2})Az_n\|^2}$ with $\sigma_n\in(0, 2)$. Otherwise, $\gamma_n=0$.
\end{description}

\subsection{The strong convergence of inertial hybrid projection algorithm}
{\bf{Algorithm 3.1}} Given appropriate parameter sequences $\{\alpha_n\}$, $\{\beta_n\}$ and $\{\gamma_n\}$, for any $x_0,\ x_1\in H_1$, the sequence $\{x_n\}$ is constructed by the following iterative form.
\begin{equation*}
\left\{\begin{aligned}
&{z_{n}=x_{n}+\alpha_{n}(x_{n}-x_{n-1}),}\\
&{u_n=J_{\beta_n}^{B_1}\left(z_n-\gamma_nA^*(I-J_{\beta_n}^{B_2})Az_n\right),}\\
&{C_{n}=\{x\in H_1:\|u_{n}-x\|^2\leq \|z_{n}-x\|^2-\theta_n\},}\\
&{Q_{n}=\{x\in H_1:\langle x_n-x_1,x_n-x \rangle\leq 0\},}\\
&{x_{n+1}=P_{C_{n}\bigcap Q_{n}}x_1,\ n\geq1,}
\end{aligned}\right.
\end{equation*}
where
\begin{equation*}
\theta_n=\gamma_n\left(2\|(I-J_{\beta_n}^{B_2})Az_n\|^2-\gamma_n\|A^*(I-J_{\beta_n}^{B_2})Az_n\|^2\right).
\end{equation*}

\begin{lemma}\label{lem1}
Assumed that (C1)-(C2) hold. For any $\gamma_n>0$, $\beta_n>0$ and set $u_n=J_{\beta_n}^{B_1}(z_n-\gamma_nA^*(I-J_{\beta_n}^{B_2})Az_n)$. Then,
\[
\|u_n-x\|^2\leq\|z_n-x\|^2-\gamma_n\left(2\|(I-J_{\beta_n}^{B_2})Az_n\|^2-\gamma_n\|A^*(I-J_{\beta_n}^{B_2})Az_n\|^2\right),\ \forall x\in \Omega,\ n\geq 1.
\]
\begin{proof}
For any $x\in \Omega$, we have $x\in B_1^{-1}(0)$ and $Ax\in B_2^{-1}(0)$. According to the property of firmly nonexpansive mappings $J_{\beta_n}^{B_1}$, $J_{\beta_n}^{B_2}$ and $I-J_{\beta_n}^{B_2}$, we have
\[
\begin{aligned}
\|u_{n} -x\|^2
&=\|J_{\beta_n}^{B_1}(z_n-\gamma_nA^*(I-J_{\beta_n}^{B_2})Az_n)-x\|^2\\
&\leq\|z_n-\gamma_nA^*(I-J_{\beta_n}^{B_2})Az_n-x\|^2\\
&=\|z_n-x\|^2+\gamma_n^2\|A^*(I-J_{\beta_n}^{B_2})Az_n\|^2-2\gamma_n\langle z_n-x, A^*(I-J_{\beta_n}^{B_2})Az_n\rangle\\
&=\|z_n-x\|^2+\gamma_n^2\|A^*(I-J_{\beta_n}^{B_2})Az_n\|^2-2\gamma_n\langle Az_n-Ax, (I-J_{\beta_n}^{B_2})Az_n-(I-J_{\beta_n}^{B_2})Ax\rangle\\
&\leq\|z_n-x\|^2+\gamma_n^2\|A^*(I-J_{\beta_n}^{B_2})Az_n\|^2-2\gamma_n\|(I-J_{\beta_n}^{B_2})Az_n\|^2\\
&=\|z_n-x\|^2-\gamma_n\left(2\|(I-J_{\beta_n}^{B_2})Az_n\|^2-\gamma_n\|A^*(I-J_{\beta_n}^{B_2})Az_n\|^2\right).
\end{aligned}
\]\qed
\end{proof}
\end{lemma}

\begin{theorem}\label{sft1}
Assumed that (C1)-(C2) and (P1)-(P2) hold. If the solution set $\Omega$ is nonempty, the iterative sequence $\{x_n\}$ generated by {Algorithm 3.1} converges strongly to $x^*=P_\Omega x_1\in \Omega$.
\end{theorem}
\begin{proof}
Firstly, we show that $P_{C_{n}\bigcap Q_{n}}$ is well defined and $\Omega\subset {C_{n}\bigcap Q_{n}}$.

From the definition of $C_n$ and $Q_n$, it is obvious that the sets $C_{n}$, $Q_{n}$ are convex and closed, which implies that $P_{C_{n}\bigcap Q_{n}}$ is well defined. For any $p\in \Omega$, it follows from Lemma \ref{lem1} that $\Omega\subset C_{n}$. In addition, $ Q_1=\{x\in H_1:\langle x_1-x_1,x_1-x \rangle\leq 0\}=H_1$, then $\Omega\subset Q_{1}$. Further, suppose $\Omega\subset {C_{n-1}\bigcap Q_{n-1}}$, using the property of metric projection and $x_{n}=P_{C_{n-1}\bigcap Q_{n-1}}x_1$, we get
\begin{equation*}
\langle x_{n}-x_1, x_{n}-x\rangle\leq 0 \,,\ \forall x\in C_{n-1}\cap Q_{n-1};
\end{equation*}
\begin{equation*}
\langle x_{n}-x_1, x_{n}-p\rangle\leq 0 \,,\ \forall p\in \Omega.
\end{equation*}
This implies that $\Omega\subset Q_{n}$. Hence, $\Omega\subset C_{n}\bigcap Q_{n}$, $n\geq 1$.

Afterwards, we show that iterative sequence $\{x_n\}$ is bounded and $\|x_{n+1}-x_n\|\rightarrow 0$ as $n\rightarrow \infty$.

Since $\Omega$ is a nonempty closed convex set, there exists a point $x^*=P_\Omega x_1\in \Omega$. Combining $x_{n+1}=P_{C_n\cap Q_n}x_1$ with $\Omega\subset C_n\cap Q_{n}$, we have $\|x_1-x_{n+1}\|\leq\|x_1-x^*\|$. Accordingly, the sequence $\{\|x_1-x_n\|\}$ is bounded, i.e., the sequence $\{x_n\}$ is bounded. From the definition of $Q_n$ and $x_{n+1}=P_{C_n\cap Q_n}x_1\in Q_n$, we get $x_n=P_{Q_n}x_1$ and $\|x_1-x_n\|\leq\|x_1-x_{n+1}\|$. These indicate that $\lim_{n\rightarrow \infty}\|x_1-x_n\|$ exists. Further, it follows from the property of metric projection $P_{Q_n}$ that
\[
\|x_n-x_{n+1}\|^2\leq\|x_1-x_{n+1}\|^2-\|x_1-x_n\|^2.
\]
This implies $\lim_{n\rightarrow \infty}\|x_n-x_{n+1}\|=0$.

Lastly, we prove that the sequence $\{x_n\}$ converges strongly to $x^*=P_\Omega x_1$.

From the boundedness of $\{x_n\}$, there exists a subsequence $\{x_{n_l}\}$  of $\{x_n\}$ converges weakly to $q$, for any $q\in \omega_w(x_n)$. Furthermore, $\|z_n-x_n\|=\alpha_{n}\|x_{n}-x_{n-1}\|\rightarrow 0$, as $n\rightarrow \infty$. This implies that $\{z_n\}$ is bounded and $z_{n_l}\rightharpoonup q$. From (P2) and Algorithm 3.1, we have $\|u_n-x_{n+1}\|^2\leq\|z_n-x_{n+1}\|^2-\theta_n\leq\|z_n-x_{n+1}\|^2$. In addition,
\[
\begin{aligned}
\|u_n-z_n\|&\leq\|u_n-x_{n}\|+\|x_n-z_{n}\|\\
&\leq\|u_n-x_{n+1}\|+\|x_n-x_{n+1}\|+\|x_n-z_n\|\\
&\leq 2\|z_n-x_n\|+2\|x_n-x_{n+1}\|\rightarrow 0,\  n\rightarrow\infty.
\end{aligned}
\]
Hence, the sequence $\{u_n\}$ is bounded. Using Lemma \ref{lem1}, for any $p\in \Omega$,
\begin{align*}
\theta_n&\leq \|z_{n}-p\|^2-\|u_{n}-p\|^2\\
&\leq (\|z_n-p\|-\|u_n-p\|)(\|z_n-p\|+\|u_n-p\|)\\
&\leq \|z_n-u_n\|(\|z_n-z\|+\|u_n-p\|)\rightarrow 0,\ n\rightarrow\infty.
\end{align*}
If $Az_n\notin B_2^{-1}0$, from the definition of $\theta_n$, $\lim_{n\rightarrow \infty}\|(I-J_{\beta_n}^{B_2})Az_n\|=0$. On the other hand, from the definition of $u_n$ and the firmly nonexpansive property of $J_{\beta_n}^{B_1}$, we obtain
\[
\|u_n-J_{\beta_n}^{B_1}z_n\|\leq\|\gamma_nA^*(I-J_{\beta_n}^{B_2})Az_n\|\leq\gamma_n\|A\|\|(I-J_{\beta_n}^{B_2})Az_n\|\rightarrow 0,\ \text{as}\ n\rightarrow\infty.
\]
Therefore, we also have $\lim_{n\rightarrow\infty}\|z_n-J_{\beta_n}^{B_1}z_n\|=0$. Since $A$ is a bounded linear operator, we get $Az_{n_l}\rightharpoonup Aq$. By Remark \ref{remark21} and Lemma \ref{lem2}, it follows that $q\in Fix(J_{\beta_n}^{B_1})$ and $Aq\in Fix(J_{\beta_n}^{B_2})$, that is, $q\in \Omega$. Meanwhile, if $Az_n\in B_2^{-1}0$, we can also get the same result. In summary, we have $\omega_w(x_n)\subset\Omega$ and $\|x_n-x_1\|\leq\|x^*-x_1\|$. By virtue of Lemma \ref{lem3}, we obtain that $\{x_n\}$ converges strongly to $x^*=P_\Omega x_1$.\qed
\end{proof}

\subsection{The strong convergence of inertial shrinking projection algorithms}
{\bf{Algorithm 3.2}} Given appropriate parameter sequences $\{\alpha_n\}$, $\{\beta_n\}$ and $\{\gamma_n\}$, for any $x_0,\ x_1\in H_1$, the sequence $\{x_n\}$ is constructed by the following iterative process.
\begin{equation*}
\left\{\begin{aligned}
&{z_{n}=x_{n}+\alpha_{n}(x_{n}-x_{n-1}),}\\
&{u_n=J_{\beta_n}^{B_1}\left(z_n-\gamma_nA^*(I-J_{\beta_n}^{B_2})Az_n\right),}\\
&{x_{n+1}=P_{C_{n+1}}x_1, n\geq 1,}
\end{aligned}\right.
\end{equation*}
where
\begin{equation*}
C_{n+1}=\left\{x\in C_n: \|u_n-x\|^2\leq\|z_n-x\|^2-\gamma_n\left(2\|(I-J_{\beta_n}^{B_2})Az_n\|^2-\gamma_n\|A^*(I-J_{\beta_n}^{B_2})Az_n\|^2\right)\right\}.
\end{equation*}
{\bf{Algorithm 3.3}} Given appropriate parameter sequences $\{\alpha_n\}$, $\{\beta_n\}$ and $\{\gamma_n\}$, for any $x_0,\ x_1\in H_1$, the sequence $\{x_n\}$ is constructed by the following iterative process.
\begin{equation*}
\left\{\begin{aligned}
&{z_{n}=x_{n}+\alpha_{n}(x_{n}-x_{n-1}),}\\
&{u_n=J_{\beta_n}^{B_1}\left(z_n-\gamma_nA^*(I-J_{\beta_n}^{B_2})Az_n\right),}\\
&{x_{n+1}=P_{C_{n+1}}x_n, n\geq 1,}
\end{aligned}\right.
\end{equation*}
where
\begin{equation*}
C_{n+1}=\left\{x\in C_n: \|u_n-x\|^2\leq\|z_n-x\|^2-\gamma_n\left(2\|(I-J_{\beta_n}^{B_2})Az_n\|^2-\gamma_n\|A^*(I-J_{\beta_n}^{B_2})Az_n\|^2\right)\right\}
\end{equation*}

\begin{theorem}\label{sft2}
Assumed that (C1)-(C2) and (P1)-(P2) hold. If the solution set $\Omega$ is nonempty, the iterative sequence $\{x_n\}$ generated by {Algorithm 3.2} converges strongly to $x^*=P_\Omega x_1\in \Omega$.
\end{theorem}

\begin{proof}
Firstly, it is obvious that the half space $C_n\ (n\geq 1)$ is convex and closed and $P_{C_n}$ is well defined. By Lemma \ref{lem1}, we can easily get that the solution set $\Omega\subset C_n$. Using $x_n=P_{C_n}x_1$, $x_{n+1}=P_{C_{n+1}}x_1$ and $C_{n+1}\subset C_n$, we have $\|x_n-x_1\|\leq\|x_{n+1}-x_1\|$, which implies that $\{\|x_n-x_1\|\}$ is nondecreasing. Furthermore, $\|x_n-x_1\|\leq\|p-x_1\|$, for any $p\in \Omega$, that is, $\{x_n\}$ is bounded. These imply that $\lim_{n\rightarrow \infty}\|x_n-x_1\|$ exists. According to the proof in Theorem \ref{sft1}, we also prove that the sequence $\{x_n\}$ converges strongly to $x^*=P_{\Omega}x_1$.\qed
\end{proof}

\begin{theorem}\label{sft3}
Assumed that (C1)-(C2) and (P1)-(P2) hold. If the solution set $\Omega$ is nonempty, the iterative sequence $\{x_n\}$ generated by {Algorithm 3.3} converges strongly to $x^*=P_\Omega x_1\in \Omega$.
\end{theorem}
\begin{proof}
Similarly, we obtain that $C_n\ (n\geq 1)$ is convex and closed, $P_{C_n}$ is well defined and $\Omega\subset C_n$. By $x_n=P_{C_n}x_{n-1}$, $x_{n+1}=P_{C_{n+1}}x_n$ and $C_{n+1}\subset C_n$, we have $\|x_n-x_1\|\leq\|x_{n+1}-x_1\|$ and $\|x_n-x_1\|\leq\|p-x_1\|,\ \forall p\in \Omega$. Using the proof in Theorems \ref{sft1} and \ref{sft2}, we have that $\{x_n\}$ converges strongly to $x^*=P_{\Omega}x_1$.\qed
\end{proof}
\section{Theoretical applications}\label{sec4}
In this section, we give several interesting special cases of the split variation inclusion problem \eqref{svip}. At the same time, Algorithms 3.1, 3.2 and 3.3 are applied to these problems. Further, the same strong convergence property in Theorems \ref{sft1}, \ref{sft2} and \ref{sft3} are proved.
\subsection{Split variational inequality problem}\label{sec4.1}
Let $C$ and $Q$ be nonempty closed convex subsets of Hilbert spaces $H_1$ and $H_2$, respectively. Let $F:H_1\rightarrow H_1$ and $G:H_2\rightarrow H_2$ be given operators, $A:H_1\rightarrow H_2$ be a bounded linear operator. The split variational inequality problem is to find a point $x^*\in C$ such that
\[
\langle F(x^*), x-x^*\rangle\geq 0,\ \forall x\in C\ \text{and}\ \langle G(Ax^*), y-Ax^*\rangle\geq 0,\ \forall y\in Q.
\]
Especially, when $H_1=H_2$, $F=G$ and $A=I$, the split variational inequality problem is transformed into the classical variational inequality problem which is to find a point $x^*\in C$ such that $\langle F(x^*), x-x^*\rangle\geq 0,\ \forall x\in C$. Hence, the solution set of the variational inequality problem is represented by $VI(F,C)$. Then, the split variational inequality problem is formulated as
\[
\text{find}\ x^*\in C\ \text{such that}\ x^*\in VI(F,C)\ \text{and}\ Ax^*\in VI(G,Q).
\]
To solve the variational inequality problem, the normal cone $N_C(x)$ of $C$ at a point $x\in C$ is defined as follows:
\[
N_C(x)=\{z\in H: \langle z, v-x\rangle\leq 0,\ \forall v\in C\}.
\]
Further, the set valued mapping $S_F$ related to the normal cone $N_C(x)$ is defined by
\[
S_F(x):=\left\{\begin{array}{cc}F(x)+N_{C}(x), & x \in C, \\ \emptyset, & \text { otherwise. }\end{array}\right.
\]
In the sense, if $F$ is a $\alpha$-inverse strongly monotone operator (i.e., for any $x,z\in C$, $\langle F(x)-F(z), x-z\rangle\geq \alpha\|F(x)-F(z)\|^2$), then $S_F$ is a maximal monotone mapping. More importantly, $x\in VI(F, C)$ if and only if $0\in S_F(x)$. Consequently, let $F$ and $G$ be $\alpha$-inverse strongly monotone operators. The set valued mappings $S_F$ and $S_G$ are associated with $F$ and $G$, respectively. In \ref{svip}, when $B_1=S_F$ and $B_2=S_G$, we obtain the above split variational inequality problem.

\subsection{Split saddle point problem}\label{sec4.2}
Let $\mathcal{X}$ and $\mathcal{Y}$ be Hilbert spaces. A bifunction $L:\mathcal{X}\times \mathcal{Y}\rightarrow \mathbb{R}\cup\{-\infty, \infty\}$ is convex-concave if and only if $L(x,\cdot)$ is convex for any $x\in \mathcal{X}$ and $L(\cdot, y)$ is concave for any $y\in \mathcal{Y}$. The operator $T_L$ is defined as follows:
\[
T_L(x,y)=(\partial_1L(x,y), \partial_2(-L)(x,y)),
\]
where $\partial_1$ is the subdifferential of $L$ with respect to $x$ and  $\partial_2$ is the subdifferential of $-L$ with respect to $y$. It is worth noting that $T_L$ is maximal monotone if and only if $L$ is closed and proper, for detail, see, \cite{Rockafellarmonotone}. Naturally, the zeros of $T_L$ coincide with the saddle points of $L$. Therefore, let $\mathcal{X}_i (i=1, 2)$, $\mathcal{Y}_i\ (i=1, 2)$ be Hilbert spaces. Let $A:\mathcal{X}_1\times \mathcal{Y}_1\rightarrow \mathcal{X}_2\times \mathcal{Y}_2$ be a bounded linear operator with adjoint operator $A^*$. Let $L_1$ and $L_2$ be closed proper convex-concave bifunctions. Then, the split saddle point problem is to find a point $(x^*, y^*)\in \mathcal{X}_1\times \mathcal{Y}_1$ such that
\[
(x^*,y^*)\in \operatorname{argminmax}_{(x,y)\in \mathcal{X}_1\times \mathcal{Y}_1}L_1(x,y)\ \text{and}\ A(x^*,y^*)\in \operatorname{argminmax}_{(z,w)\in \mathcal{X}_2\times \mathcal{Y}_2}L_2(z,w).
\]
In other words, when $H_i=\mathcal{X}_i\times \mathcal{Y}_i\ (i=1, 2)$, $B_i=T_{L_i}\ (i=1, 2)$, the split variational inclusion problem is reduced to the split saddle point problem.

\subsection{Split minimization problem}\label{sec4.3}
Let $H_1$ and $H_2$ be Hilbert spaces. Let $\phi:H_1\rightarrow \mathbb{R}$ and $\psi:H_2\rightarrow \mathbb{R}$ be lower semi-continuous convex functions, $A:H_1\rightarrow H_2$ be a bounded linear operator. The split minimization problem is to find $x^*\in H_1$ such that
\[
x^*\in \operatorname{argmin}_{x\in H_1}\phi(x)\ \text{and}\ Ax^*\in \operatorname{argmin}_{y\in H_2}\psi(y).
\]
As we all know, $x^*\in \operatorname{argmin}_{x\in H_1}\phi(x)$ if and only if $0\in \partial \phi(x^*)$, where $\partial \phi$ is the subdifferential of $\phi$ defined by $\partial \phi(x^*):=\left\{\hat{x} \in H_1: \phi(x^*)+\left\langle z-x^*, \hat{x}\right\rangle \leq \phi(z),\ \forall z \in H_1\right\}$.  Recall that the proximal operator $\operatorname{prox}_\phi$ of $\phi$ is as follows:
\[
\operatorname{prox}_\phi x=\operatorname{argmin}_{z\in H_1}\left\{\phi(z)+\frac{1}{2\gamma}\|z-x\|^2\right\},\ \forall \gamma>0.
\]
It is very important that $\operatorname{prox}_{\phi}(x)=(I+\gamma \partial \phi)^{-1}(x)=J_{\gamma}^{\partial \phi}(x)$. In addition, $\partial \phi$ is a maximal monotone mapping and $\operatorname{prox}_{\phi}$ is a firmly nonexpansive mapping. In view of this, when $B_1=\partial \phi$ and $B_2=\partial \psi$ in \eqref{svip}, the split variational inclusion problem is transformed into the split minimization problem.

\begin{remark}
Through the above results, the split variational inclusion problem is transformed into other problems, such as the split variational inequality problem, the split saddle point problem and the split minimization problem. Using the same algorithms and techniques in Theorems \ref{sft1}, \ref{sft2} and \ref{sft3}, the strong convergence property of these problems are obtained under the above corresponding conditions in Subsections \ref{sec4.1}, \ref{sec4.2} and \ref{sec4.3}.
\end{remark}

\section{Numerical example}\label{sec5}
In this section, a numerical example is provided to illustrate the effectiveness and realization of convergence behavior of Algorithms 3.1, 3.2 and 3.3. All codes were written in Matlab R2018b, and ran on a Lenovo ideapad 720S with 1.6 GHz Intel Core i5 processor and 8GB of RAM. Our results compare the existing conclusions below. Firstly, let $H_1$ and $H_2$ be Hilbert spaces, $A:H_1\rightarrow H_2$ be a bounded linear operator with the adjoint operator $A^*$. Let $B_1: H_1\rightarrow 2^{H_1}$ and $B_2: H_2\rightarrow 2^{H_2}$ be two set-valued maximal monotone mappings. Many existing conclusions on the split variational inclusion problem have been proven in such an environment as follows.

\begin{theorem} \emph{(Byrne et al. \cite[Algorithm 4.4]{byrne2011weak})}
For any initial point $x_1\in H_1$, $\delta_n\in (0,1)$ and $\beta>0$, the iterative sequence $\{x_n\}$ is generated by the following iterative scheme
\[
x_{n+1}=\delta_n x_1+(1-\delta_n)J_{\beta}^{B_1}\left(x_n-\gamma A^*(I-J_{\beta}^{B_2})Ax_n\right),\ n\geq 1.
\]
If the sequence $\{\delta_n\}$ satisfies $\underset{n\rightarrow \infty}{\lim} \delta_n=0$ and $\sum_{n=1}^{\infty}\delta_n= \infty$, $0<\gamma<2/\|A^*A\|$, then the iterative sequence $\{x_n\}$ converges strongly to a point $x^*\in \Omega$.
\end{theorem}

\begin{theorem} \emph{(Long et al. \cite[Algorithm (49)]{long2019new})}
For any initial points $x_0, x_1\in H_1$ and $\beta>0$, the iterative sequence $\{x_n\}$ is generated by the following iterative scheme.
\[
\left\{\begin{aligned}
&{z_{n}=x_{n}+\alpha_{n}(x_{n}-x_{n-1}),}\\
&{u_n=J_{\beta}^{B_1}\left(z_n-\gamma_nA^*(I-J_{\beta}^{B_2})Az_n\right),}\\
&{x_{n+1}=\delta_nf(x_n)+(1-\delta_n)u_n, n\geq 1,}
\end{aligned}\right.
\]
where $f:H_1\rightarrow H_1$ is a contraction mapping with coefficient $k\in[0,1)$, $\{\delta_n\}$ is a sequence in $(0,1)$ such that $\underset{n\rightarrow \infty}{\lim} \delta_n=0$ and $\sum_{n=1}^{\infty}\delta_n= \infty$, $0<a\leq \gamma_n\leq b<1/\|A\|^2$, $0\leq \alpha_n\leq \alpha$ and $\underset{n\rightarrow \infty}{\lim} \frac{\alpha_n\|x_n-x_{n-1}\|}{\delta_n}=0$. The iterative sequence $\{x_n\}$ converges strongly to a point $x^*\in \Omega$.
\end{theorem}

\begin{theorem} \emph{(Anh et al. \cite[Algorithm (4)]{anh2020astrong})}
For any initial points $x_0, x_1\in H_1$ and $\beta>0$, the iterative sequence $\{x_n\}$ is generated by the following iterative scheme.
\[
\left\{\begin{aligned}
&{z_{n}=x_{n}+\alpha_{n}(x_{n}-x_{n-1}),}\\
&{u_n=J_{\beta}^{B_1}\left(z_n-\gamma_nA^*(I-J_{\beta}^{B_2})Az_n\right),}\\
&{x_{n+1}=(1-\delta_n-\theta_n)x_n+\delta_nu_n, n\geq 1,}
\end{aligned}\right.
\]
where $\{\theta_n\}$ is a sequence in $(0,1)$ with $\underset{n\rightarrow \infty}{\lim} \theta_n=0$, $\sum_{n=1}^{\infty}\theta_n= \infty$, $0<a\leq \gamma_n\leq b<1/\|A\|^2$, $0\leq \alpha_n< \alpha$ for some $\alpha>0$ with $\underset{n\rightarrow \infty}{\lim} \frac{\alpha_n\|x_n-x_{n-1}\|}{\theta_n}=0$, $0<c<\delta_n<d<1-\theta_n$. The iterative sequence $\{x_n\}$ converges strongly to a point $x^*\in \Omega$.
\end{theorem}

\begin{example}\label{ex1}
Assume that $A, A_{1}, A_{2}: \mathbb{R}^{m} \rightarrow \mathbb{R}^{m}$ are created from a normal distribution with mean zero and unit variance. Let $B_{1}:\mathbb{R}^{m} \rightarrow \mathbb{R}^{m}$ and $B_{2}: \mathbb{R}^{m} \rightarrow \mathbb{R}^{m}$ be defined by $B_{1}(x)=A_{1}^{*} A_{1} x$ and $B_{2}(y)=A_{2}^{*} A_{2} y$, respectively. Consider the problem of finding a point $\bar{x}=\left(\bar{x}_{1}, \ldots, \bar{x}_{m}\right)^{T} \in \mathbb{R}^{m}$ such that $B_{1}(\bar{x})=(0, \ldots, 0)^{T}$ and $B_{2}(A \bar{x})=(0, \ldots, 0)^{T}$. It is easy to see that the minimum norm solution of the mentioned above problem is $x^{*}=(0, \ldots, 0)^{T}$.  Our parameter settings are as follows. In our algorithms~3.1--3.3, set $ \alpha_{n}=0.5 $, $ \beta_{n}=1 $ and $ \sigma_{n}=1.5 $. Take $ \beta=1 $, $ \delta_{n}=\frac{1}{n+1} $ and $ \gamma_{n}=\frac{1.5}{\|A^{*}A\|} $ in the Algorithm~4.4 proposed by Byrne et al.~\cite{byrne2011weak}. Put $ \beta=1 $, $ \delta_{n}=\frac{1}{n+1} $, $ \alpha=0.5 $, $ \alpha_{n}=\frac{1}{(n+1)^{3}\|x_{n}-x_{n-1}\|} $, $ \gamma_{n}=\frac{0.5}{\|A^{*}A\|} $ and $ f(x)=0.8x $ in Long et al. \cite[Algorithm (49)]{long2019new}. In Anh et al. \cite[Algorithm (4)]{anh2020astrong}, choose $ \beta=1 $, $ \theta_{n}=\frac{1}{n+1} $, $ \delta_{n}=0.2(1-\theta_{n}) $, $ \alpha=0.5 $, $ \alpha_{n}=\frac{1}{(n+1)^{3}\|x_{n}-x_{n-1}\|} $ and $ \gamma_{n}=\frac{0.5}{\|A^{*}A\|} $. We use $ E_{n}=\left\|x_{n}-x^{*}\right\| $ to measure the iteration error of all algorithms. The stopping condition is $ E_{n}<\epsilon $ or the maximum number of iterations is $ 300 $ times. First, choose $ \epsilon=10^{-2},10^{-3},10^{-4},10^{-5} $. We test the convergence behavior of all algorithms under different stopping conditions. The numerical results are shown in Table~\ref{tab1} and Figure~\ref{fig1-4}. Second, Figure~\ref{fig5-8} describes the numerical behavior of all algorithms in different dimensions under the same stopping criterion $ \epsilon=10^{-4} $.
\begin{table}[H]
	\centering
	\caption{The number of termination iterations of all algorithms under different stopping criteria}
	\begin{tabular}{ccccc}
		\toprule
		\multirow{2}[2]{*}{The Algotithms} & $\epsilon= 10^{-2} $ & $  \epsilon= 10^{-3} $ & $\epsilon=  10^{-4} $ & $ \epsilon=10^{-5} $ \\
		\cmidrule{2-5}          & iter. & iter. & iter. & iter. \\
		\midrule
		Our Alg. 3.3 & \textbf{13}    & \textbf{18}    & \textbf{23}    & \textbf{28} \\
		Our Alg. 3.2 & 248   & 300   & 300   & 300 \\
		Our Alg. 3.1 & 300   & 300   & 300   & 300 \\
		Byrne et al. Alg. 4.4 & 300   & 300   & 300   & 300 \\
		Long et al. Alg. (49) & 26    & 38    & 48    & 54 \\
		Anh et al. Alg. (4) & 37    & 80    & 131   & 185 \\
		\bottomrule
	\end{tabular}%
	\label{tab1}%
\end{table}%

\begin{figure}[H]
\centering
\subfigure[$ \epsilon=10^{-2} $]{
\includegraphics[width=0.45\textwidth]{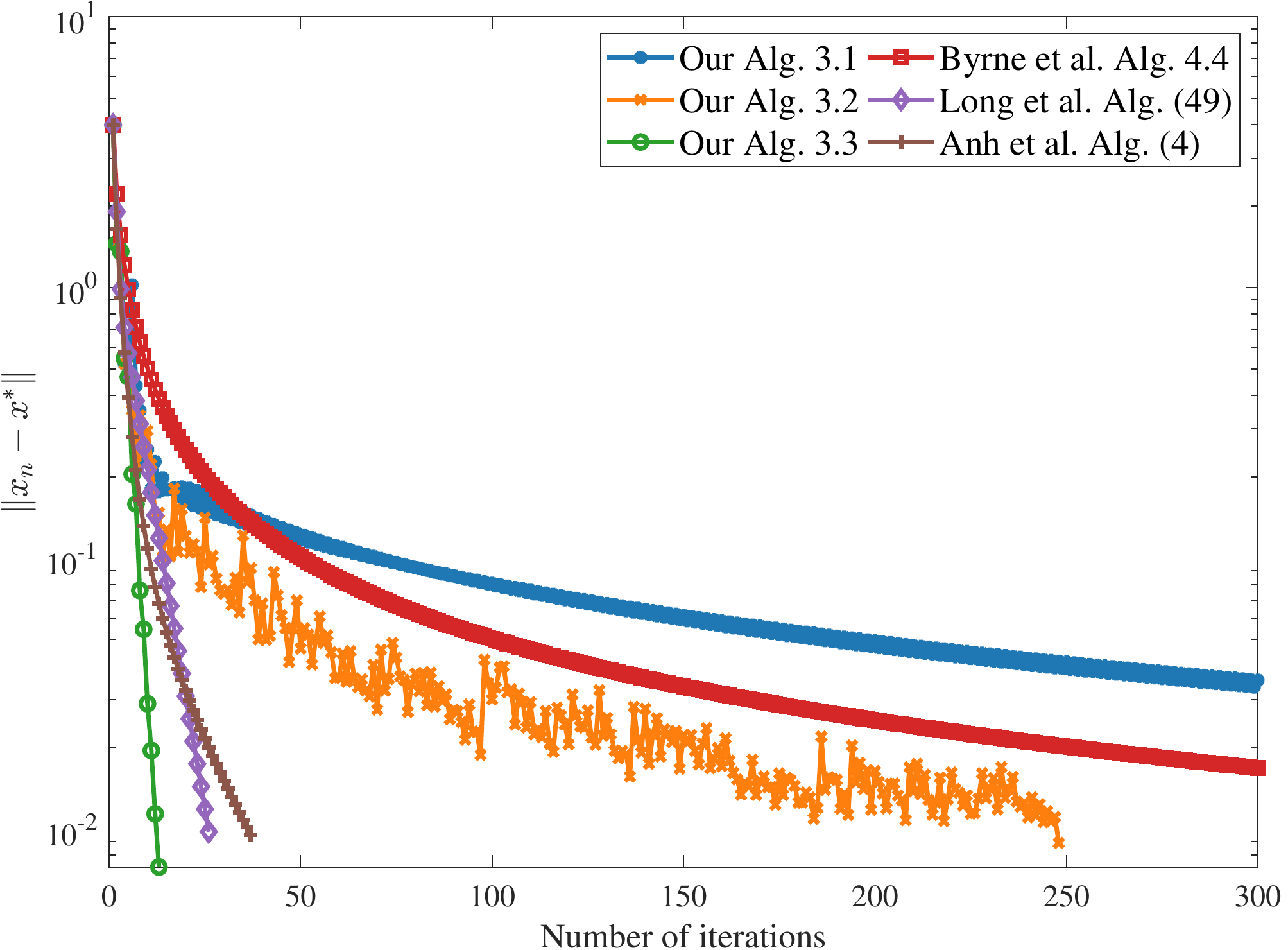}}
\subfigure[$ \epsilon=10^{-3} $]{
\includegraphics[width=0.45\textwidth]{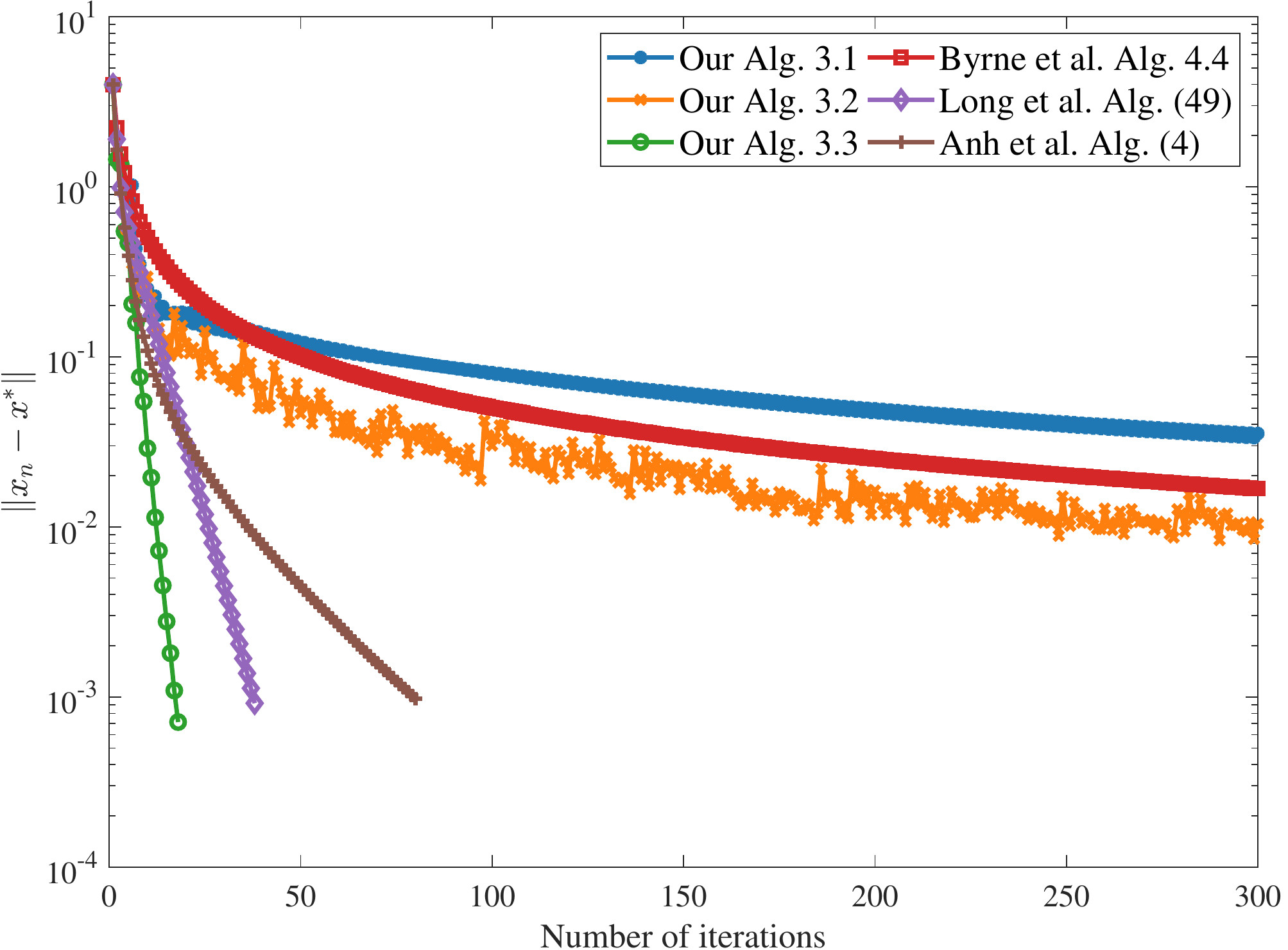}}
\subfigure[$ \epsilon=10^{-4} $]{
\includegraphics[width=0.45\textwidth]{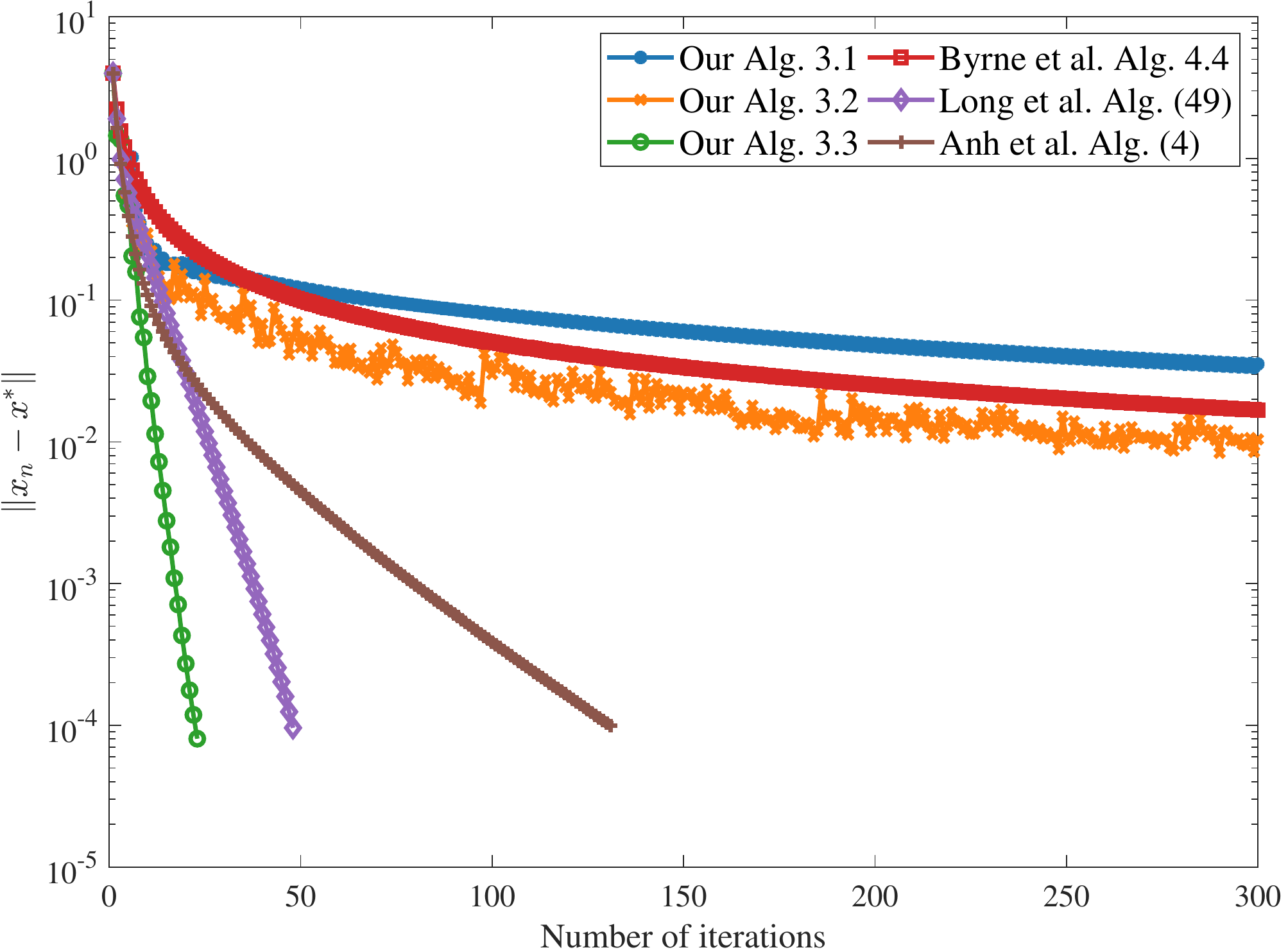}}
\subfigure[$ \epsilon=10^{-5} $]{
\includegraphics[width=0.45\textwidth]{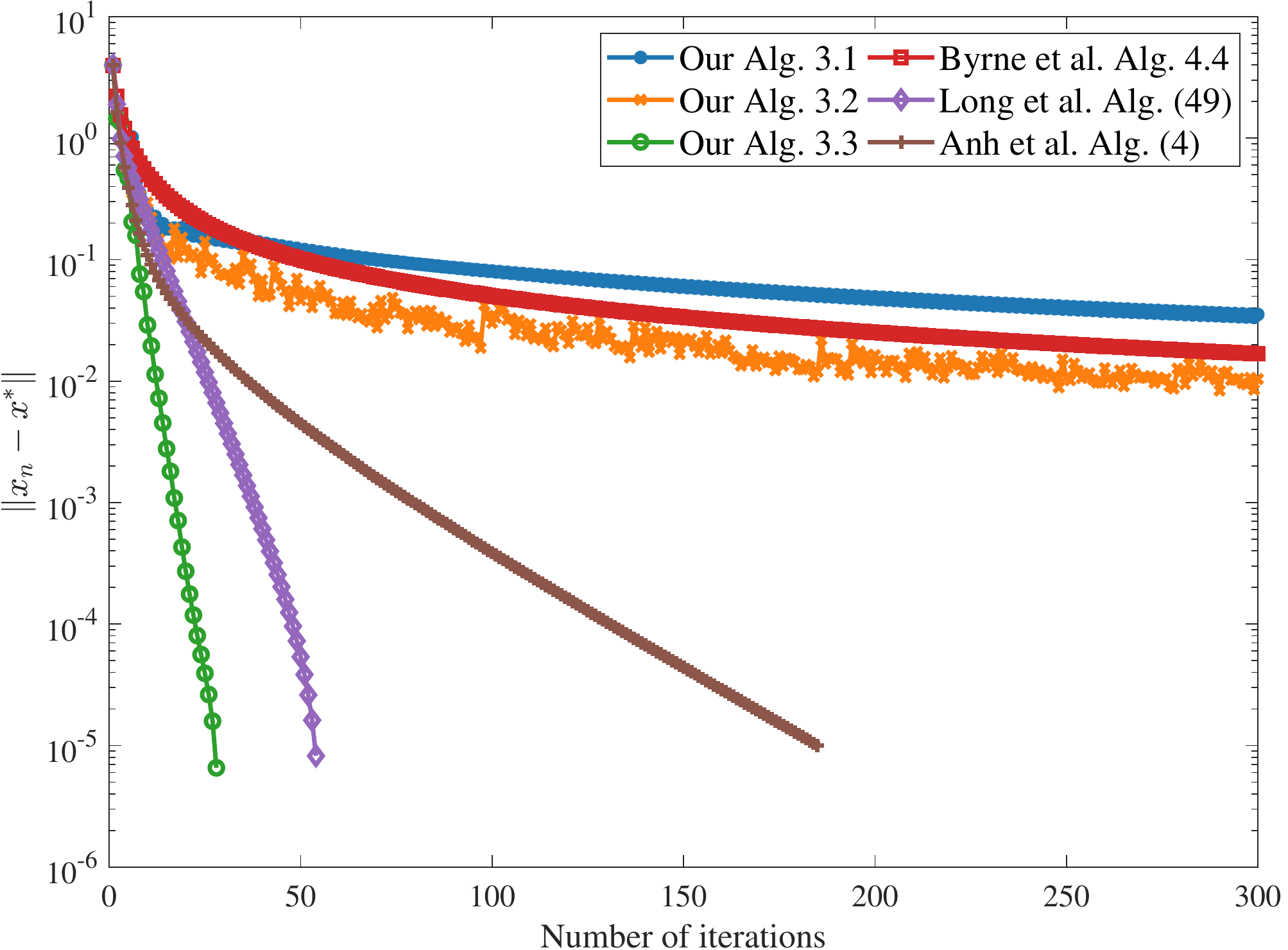}}
\caption{Numerical behavior of all algorithms under different stopping criteria}
\label{fig1-4}
\end{figure}

\begin{figure}[h]
\centering
\subfigure[$m= 60 $]{
\includegraphics[width=0.45\textwidth]{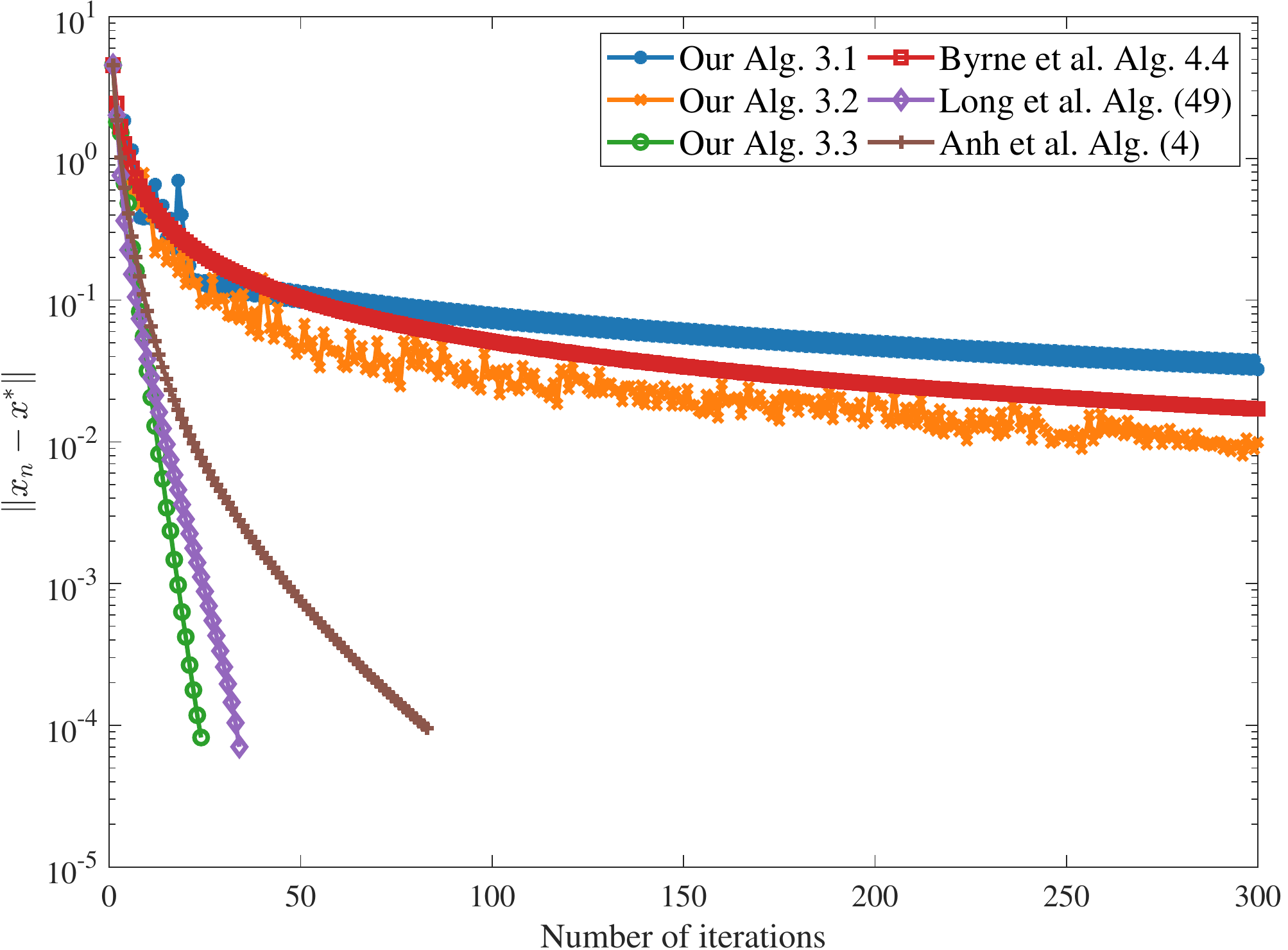}}
\subfigure[$m= 100 $]{
\includegraphics[width=0.45\textwidth]{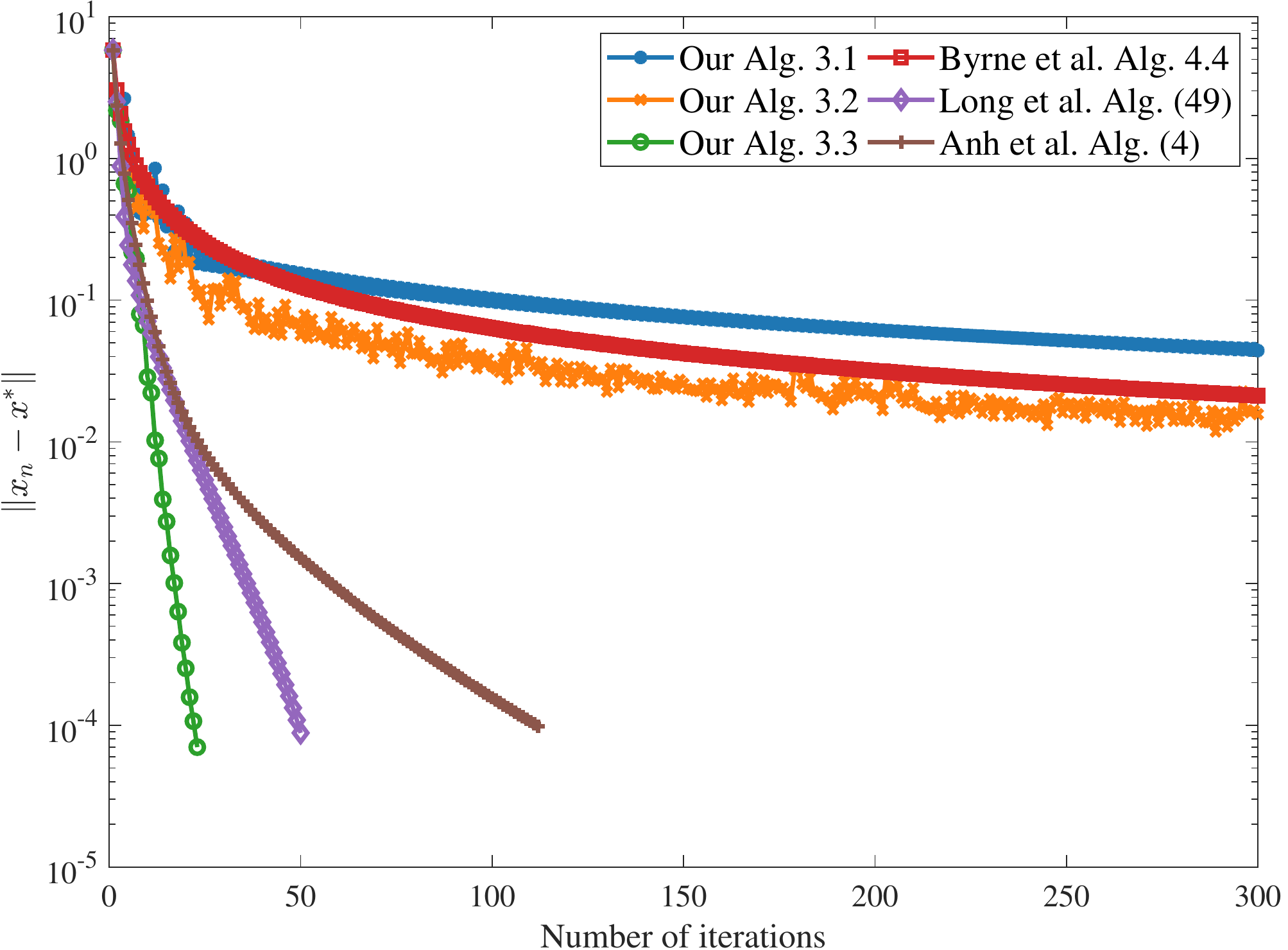}}
\subfigure[$m= 150 $]{
\includegraphics[width=0.45\textwidth]{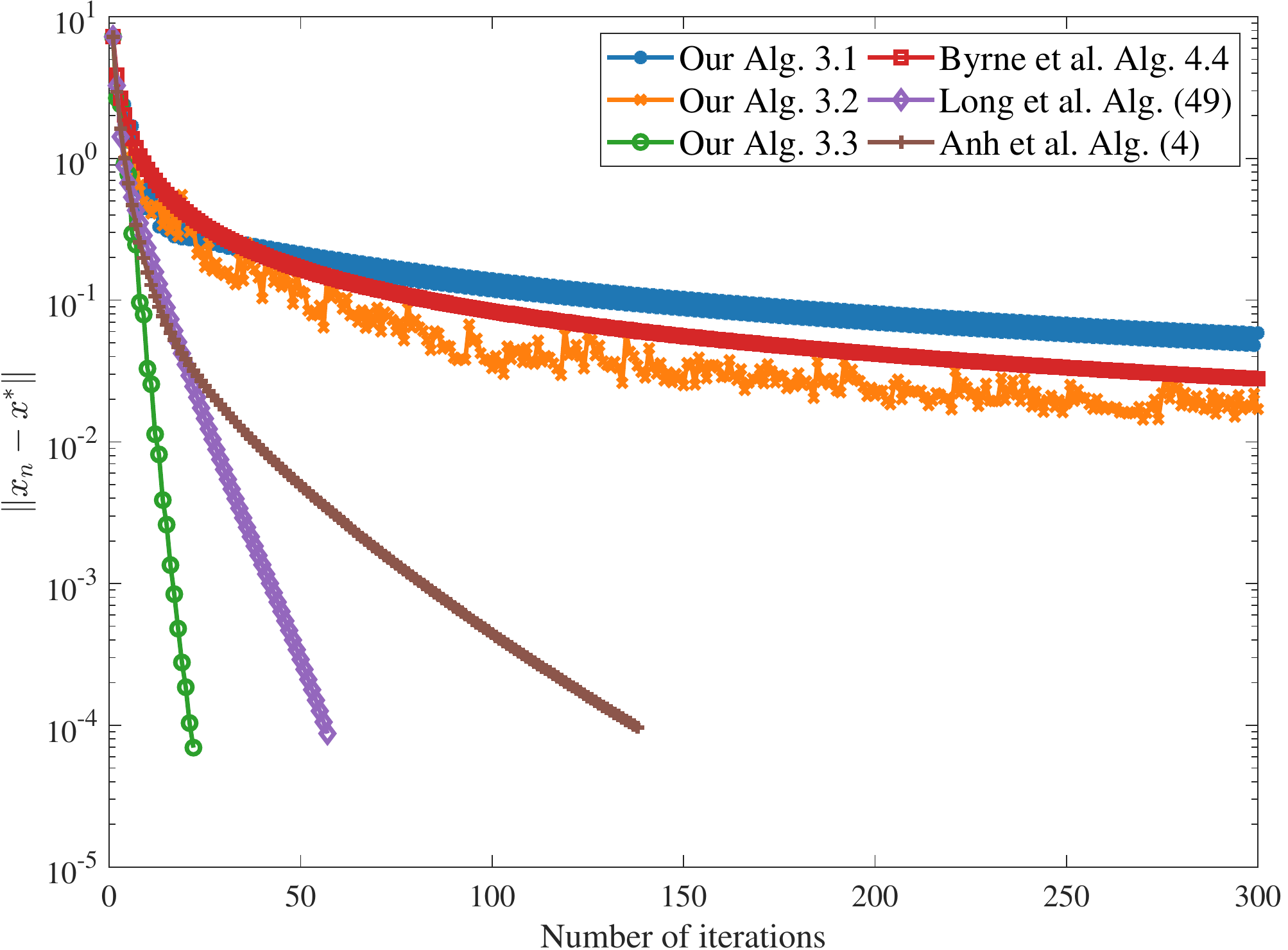}}
\subfigure[$m= 200 $]{
\includegraphics[width=0.45\textwidth]{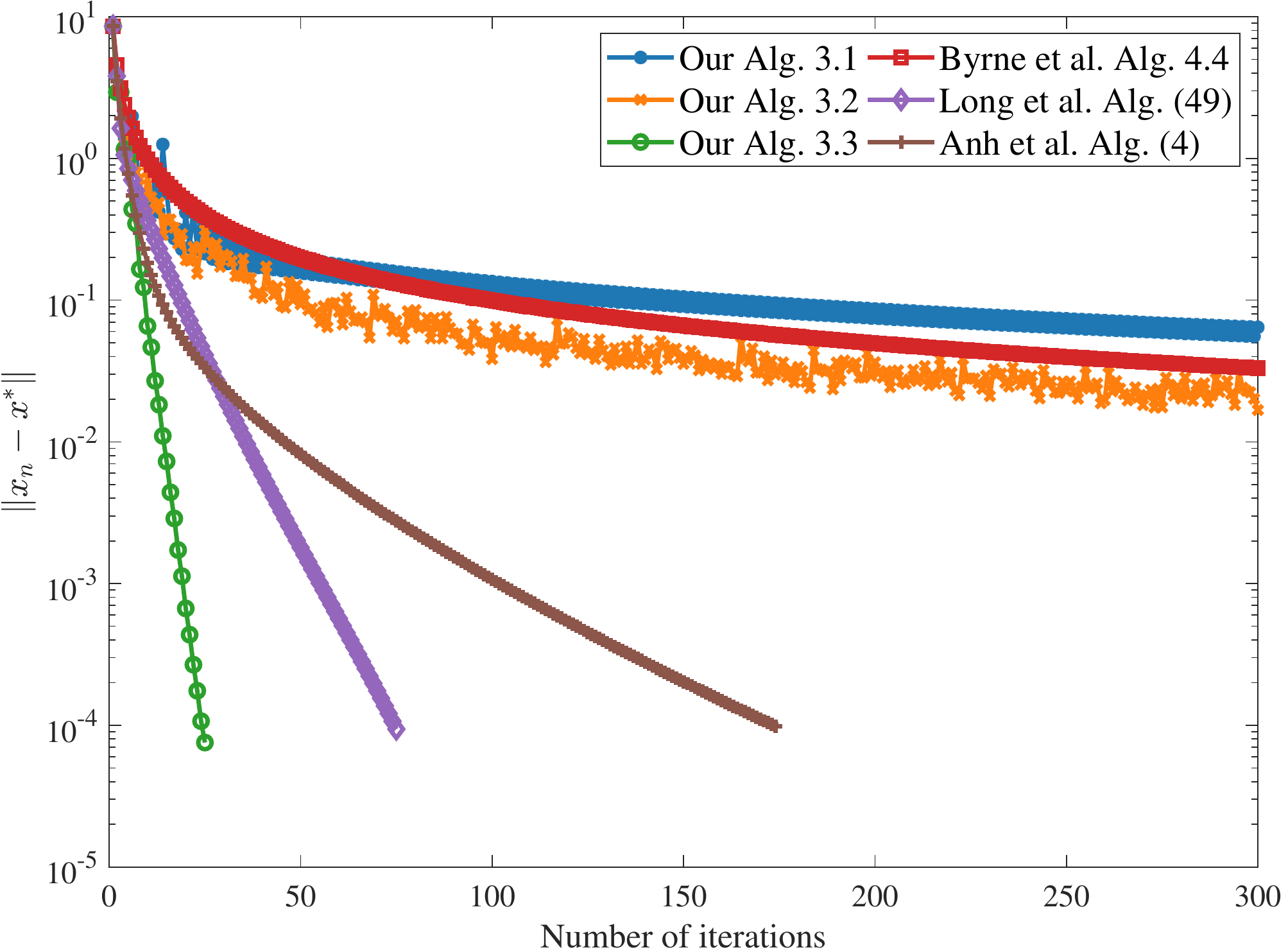}}
\caption{Numerical behavior of all algorithms in different dimensions}
\label{fig5-8}
\end{figure}

It can be seen from the above results that our Algorithms 3.1 and 3.2 are efficient and robust. These results are independent of the selection of initial values and dimensions. Moreover, note that our proposed Algorithms 3.1 and 3.2 are oscillating due to the dual reasons of inertial and projection, but the suggested Algorithm 3.3 performs very well.
\end{example}

\section{Conclusion}

In this paper, our innovation are twofold. One is to provide a self-adaptive step size selection which does not require the norm of the bounded linear operators. The other is to propose two types of projection algorithms (i.e., hybrid projection algorithms and shrinking projection algorithms), which combine inertial technique with the proposed self-adaptive step size. Under mild constraints, the corresponding strong convergence theorems of SVIP are obtained in the framework of Hilbert spaces. At the same time, our results are also extended to the split variational inequality problem, the split saddle point problem and the split minimization problem. In terms of numerical experiments, the effectiveness of our proposed algorithms are showed by comparing with some existing results.


\begin{thebibliography}{10}

\bibitem{censor2010}
Censor, Y., Gibali, A., Reich, S.:
\newblock The split variational inequality problem.
\newblock { The Technion-Israel Institue of Technology, Haifa September}
  20 (2010)

\bibitem{moudafi2011split}
Moudafi, A.:
\newblock Split monotone variational inclusions.
\newblock { J. Optim. Theory Appl.}
  150:275--283 (2011)

\bibitem{byrne2011weak}
Byrne, C., Censor, Y., Gibali, A., Reich, S.:
\newblock Weak and strong convergence of algorithms for the split common null point problem.
\newblock { J. Nonlinear Convex Anal.}
  13:759--775 (2011)

\bibitem{long2019new}
Long, L.V., Thong, D.V., Dung, V.T.:
\newblock New algorithms for the split variational inclusion problems and application to split feasibility problems.
\newblock { Optimization}
  68:2339--2367 (2019)

\bibitem{qin2019}
Qin, X., Yao, J.C.:
\newblock { A viscosity iterative method for a split feasibility problem.}
\newblock { J. Nonlinear Convex Anal.}  20:1497--1506 (2019)


\bibitem{anh2020astrong}
Anh, P.K., Thong, D.V., Dung, V.T.:
\newblock A strongly convergent Mann-type inertial algorithm for solving split variational inclusion problems.
\newblock { Optim. Eng.}
  (2020). https://doi.org/10.1007/s11081-020-09501-2

\bibitem{zhou2020anew}
Zhou, Z., Tan, B., Li, S.:
\newblock A new accelerated self-adaptive stepsize algorithm with excellent stability for split common fixed point problems.
\newblock { Comp. Appl. Math.}
  39:220 (2020)

\bibitem{chambolle1997image}
Chambolle, A., Lions, P.L.:
\newblock  Image recovery via total variation minimization and related problems.
\newblock { Numer. Math.} 76:167--188 (1997)

\bibitem{nikolova2004variational}
Nikolova, M.:
\newblock  A variational approach to remove outliers and impulse noise.
\newblock { J. Math. Imaging Vision} 20(1-2):99--120 (2004)


\bibitem{qin2014var}
Qin, X., Cho, S.Y.,  Wang, L.:
\newblock A regularization method for treating zero points of the sum of two monotone operators.
\newblock { Fixed Point Theory Appl.} 2014:75 (2014)

\bibitem{chuang2016algorithms}
Chang, S.S.,  Wen, C.F.,  Yao, J.C.:
\newblock Common zero point for a finite family of inclusion problems of accretive mappings in Banach spaces.
\newblock {  Optimization}  67:1183--1196 (2018)





\bibitem{kazmi2014an}
Kazmi, K.R., Rizvi, S.H.:
\newblock An iterative method for split variational inclusion problem and fixed point problem for a nonexpansive mapping.
\newblock { Optim. Lett.}
  8:1113--1124 (2014)



\bibitem{Sitthithakerngkiet2018}
Sitthithakerngkiet, K., Deepho, J., Martínez-Moreno, J., Kumam, P.:
\newblock Convergence analysis of a general iterative algorithm for finding a common solution of split variational inclusion and optimization problems.
\newblock { Numer. Algorithms}
  79:801--824 (2018)

\bibitem{cho2014st}
Cho, S.Y., Qin, X., Wang, L.:
\newblock Strong convergence of a splitting algorithm for treating monotone operators.
\newblock { Fixed Point Theory Appl.} 2014:94 (2014)



\bibitem{alvarez2001}
Alvarez, F., Attouch, H.:
\newblock An inertial proximal method for maximal monotone operators via discretization of a nonlinear oscillator with damping.
\newblock { Set-Valued Anal.}
  9(1-2):3–11 (2001)

\bibitem{zhou2020an}
Zhou, Z., Tan, B., Li, S.:
\newblock An inertial shrinking projection algorithm for split common fixed point problems.
\newblock { J. Appl. Anal. Comput.}
  10:2104--2120 (2020)

\bibitem{lorenz2015an}
Lorenz, D.A., Pock, T.:
\newblock An inertial forward-backward algorithm for monotone inclusions.
\newblock { J. Math. Imaging Vision}
  51:311--325 (2015)

\bibitem{qin2014v}
Qin, X., Cho, S.Y., Wang, L.:
\newblock Iterative algorithms with errors for zero points of m-accretive operators.
\newblock {Fixed Point Theory Appl.}
  2013:148 (2013)

\bibitem{tan2020strong}
Tan, B., Xu, S.:
\newblock Strong convergence of two inertial projection algorithms in Hilbert spaces.
\newblock { J. Appl. Numer. Optim.}
  2(2):171--186 (2020) 

\bibitem{NAKAJO2003372}
Nakajo, K., Takahashi, W.:
\newblock Strong convergence theorems for nonexpansive mappings and nonexpansive semigroups.
\newblock { J. Math. Anal. Appl.}
  279:372--379 (2003)

\bibitem{TAKAHASHI2008276}
Takahashi, W., Takeuchi, Y., Kubota, R.:
\newblock Strong convergence theorems by hybrid methods for families of nonexpansive mappings in Hilbert spaces.
\newblock { J. Math. Anal. Appl.}
  341:276--286 (2008)

\bibitem{marino2004convergence}
Marino, G., Xu, H.K.:
\newblock Convergence of generalized proximal point algorithm.
\newblock { Commun. Pure Appl. Anal.}
  3:791--808 (2004)

\bibitem{takahashi2000}
Zhou, H., Qin, X.:
\newblock   Fixed Points of Nonlinear Operators. Berlin, Boston: De Gruyter (2020)

\bibitem{MARTINEZYANES20062400}
Martinez-Yanes, C., Xu, H.K.:
\newblock Strong convergence of the CQ method for fixed point iteration processes.
\newblock { Nonlinear Anal.}
  64:2400--2411 (2006)

\bibitem{Rockafellarmonotone}
Rockafellar, R.T.:
\newblock Monotone operators associated with saddle functions and minimax problems. In: Browder F.E. (ed.) Nonlinear Functional Analysis, Part 1, vol. 18, pp. 397–407.
\newblock { Amer. Math. Soc.}
  (1970)

\end{thebibliography}
\end{document}